\documentclass[graybox]{svmult}

\usepackage{mathptmx}       
\usepackage{helvet}         
\usepackage{courier}        
\usepackage{makeidx}         
\usepackage{graphicx}        
\usepackage{multicol}        
\usepackage[bottom]{footmisc}
\usepackage{url}
\usepackage{amsfonts,amssymb,amsmath,amsgen}
\usepackage{hyperref,color}
\usepackage{pdfsync}
\usepackage{enumerate}

\usepackage{bbm}

\spnewtheorem*{varproof}{\noexpand\thisproofname}{\itshape}{\rmfamily}
\newcommand{\thisproofname}{}
\newenvironment{prooff*}[1]
 {\renewcommand\thisproofname{#1}\varproof}
 {\endvarproof}

\newcommand{\R}{\mathbb{R}}
\newcommand{\C}{\mathbb{C}}

\newcommand{\ii}{\mathrm{i}}

\makeindex             

\begin{document}

\title*{Zero modes and low-energy resolvent expansion for three dimensional Schr\"odinger operators with point interactions}
\titlerunning{Low-energy resolvent expansion for multi-center point interactions}
\author{Raffaele Scandone}
\institute{Raffaele Scandone \at GSSI - Gran Sasso Science Institute, Viale F. Crispi 7, 67100 L'Aquila \email{raffaele.scandone@gssi.it}}

\maketitle

\abstract*{We investigate the low-energy behavior of the resolvent of Schr\"odinger operators with finitely many point interactions in three dimensions. We also discuss the occurrence and the multiplicity of zero energy obstructions.}

\abstract{We investigate the low-energy behavior of the resolvent of Schr\"odinger operators with finitely many point interactions in three dimensions. We also discuss the occurrence and the multiplicity of zero energy obstructions.}

\section{Introduction and main results}\label{secuno}
A central topic in quantum mechanics is the study of quantum systems subject to very short-range interactions, supported around a submanifold of the ambient space. A relevant situation occurs when the \emph{singular} interaction is supported on a set of points in the Euclidean space $\R^d$. This leads to consider, formally, operators of the form
\begin{equation}\label{formalHamilt}
^\textrm{``}-\Delta+\sum_{y\in Y}\,\mu_y\,\delta_y(\cdot)^\textrm{''},
\end{equation}
where $Y$ is a discrete subset of $\R^d$, and $\mu_y$, $y\in Y$, are real coupling constants.

Heuristically, \eqref{formalHamilt} can be interpreted as the Hamiltonian for a non-relativistic quantum particle interacting with ``point obstacles'' of strengths $\mu_y$, located at $y\in Y$.

From a mathematical point of view, Schr\"odinger operators with point (delta-like) interactions have been intensively studied, since the first rigorous realization by Berezin and Faddeev \cite{Berezin-Faddeev-1961}, and subsequent characterizations by many other authors \cite{Albeverio-Fenstad-HoeghKrohn-1979_singPert_NonstAnal,Zorbas-1980,Grossmann-HK-Mebkhout-1980,Grossmann-HK-Mebkhout-1980_CMPperiodic,Dabrowski-Grosse-1985,MO-2016} (see the surveys \cite{DFT-brief_review_2008,Albeverio-Figari}, the monograph of Albeverio, Gesztesy, H{\o}egh-Krohn, and Holden \cite{albeverio-solvable}, and references therein for a thorough discussion). 

In this work we focus on the case of finitely many point interactions in three dimensions. Our aim is to provide a detailed spectral analysis at the bottom of the continuous spectrum, i.e. at zero energy. A similar analysis has been done in \cite{Cornean-MY} for the two dimensional case, with application to the $L^p$-boundedness of the wave operators. 

We start by recalling some well-known facts on the construction and the main properties of Schr\"odinger operators with point interactions.

We fix a natural number $N\geqslant 1$ and the set $Y=\{y_1,\dots,y_N\}\subseteq\R^3$ of distinct centers of the singular interactions. Consider
\begin{equation}
T_Y:=\;\overline{(-\Delta)\upharpoonright C^\infty_0(\mathbb{R}^3\!\setminus\!\{Y\})}
\end{equation}
as an operator closure with respect to the Hilbert space $L^2(\mathbb{R}^3)$. 
It is a closed, densely defined, non-negative, symmetric operator on $L^2(\mathbb{R}^3)$, with deficiency index $N$. Hence, it admits an $N^2$-real parameter family of self-adjoint extensions. Among these, there is an $N$-parameter family of \emph{local} extension, denoted by
\begin{equation}
\{-\Delta_{\alpha,Y}\;\big|\;\alpha\equiv(\alpha_1,\ldots,\alpha_N)\in(\R\cup\{\infty\})^N\},
\end{equation}
whose domain of self-adjointness is qualified by certain local boundary conditions at the singularity centers. 

The self-adjoint operators $-\Delta_{\alpha,Y}$ provide rigorous realizations of the formal Hamiltonian \eqref{formalHamilt}, the coupling parameters $\alpha_j$, $j=1,\ldots ,N$, being now proportional to the inverse scattering length of the interaction at the center $y_j$. In particular, if for some $j\in\{1,\ldots ,N\}$ one has $\alpha_j=\infty$, then no actual interaction is present at the point $y_j$, and in practice things are as if one discards the point $y_j$. When all $\alpha_j=\infty$, one recovers the Friedrichs extension of $T_Y$, namely the self-adjoint realization of $-\Delta$ on $L^2(\R^3)$. Owing to the discussion above, we may henceforth assume, without loss of generality, that $\alpha$ runs over $\R^N$.

We review the basic properties of $-\Delta_{\alpha,Y}$, from \cite[Section II.1.1]{albeverio-solvable} and \cite{Posilicano2000_Krein-like_formula} (see also \cite{DFT-brief_review_2008, DAncona-Pierfelice-Teta-2006,Iandoli-Scandone-2017,DMSY-2017}). We introduce first some notations.

For $z\in \mathbb{C}$ and $x,y,y'\in \mathbb{R}^3$, set 
\begin{equation}\label{eq:def_of_the_Gs}
\mathcal{G}_z^y(x)\;:=\;\frac{e^{\ii z|x-y|}}{\,4\pi |x-y|\,},\quad
\mathcal{G}_z^{yy'}\;:=\;\begin{cases}
\displaystyle\frac{e^{\ii z|y-y'|}}{\,4\pi |y-y'|\,} & \textrm{if }\;y'\neq y \\ 
\qquad 0 & \textrm{if }\;y'= y\,, \\
\end{cases}
\end{equation}
and 
\begin{equation}\label{ga-def}
\Gamma_{\alpha,Y}(z)\;:=\;\Big(\Big(\alpha_j-\frac{\ii z}{\,4\pi\,}\Big)\delta_{j,k}-\mathcal{G}_z^{y_jy_k}\Big)_{\!j,k=1,\dots,N}\,.
\end{equation}
The function $z\mapsto \Gamma_{\alpha,Y}(z)$ has values in the space of $N\times N$ symmetric, complex valued matrices and is clearly entire, whence $z\mapsto \Gamma_{\alpha,Y}(z)^{-1}$ is meromorphic in $\mathbb{C}$. It is known that $\Gamma_{\alpha,Y}(z)^{-1}$ has at most $N$ poles in the open upper half-plane $\mathbb{C}^{+}$, which are all located along the positive imaginary semi-axis. We denote by $\mathcal{E}^+$ the set of such poles. Moreover, we denote by $\mathcal{E}^0$ the set of poles of $\Gamma_{\alpha,Y}(z)^{-1}$ on the real line. Observe that $\mathcal{E}^0$ is finite and symmetric with respect to $z=0$. Actually, either $\mathcal{E}^0=\emptyset$ or $\mathcal{E}^0=\{0\}$. This follows by a generalization of the Rellich Uniqueness Theorem \cite[Theorem 2.4]{sjo_lecture}, valid for a large class of compactly supported perturbations of the Laplacian, introduced by Sj\"ostrand and Zworski in \cite{Sjo-Zwo}. For an introduction to the classical theory of the Rellich Uniqueness Theorem, we refer to the monograph of Lax and Phillips \cite{Lax-Phi}. More recently, the absence of non-zero real poles for $\Gamma_{\alpha,Y}^{-1}$ has been proved through different techniques by Galtbayar-Yajima \cite{Galtbayar-Yajima-2019}, and by the author in collaboration with Michelangeli \cite{MS-2020}.

The following facts are known.

\begin{proposition}\label{thm:general_properties}~
\begin{enumerate}[(i)]
\item The domain of $-\Delta_{\alpha,Y}$ has the following representation, for any $z\in\mathbb{C}^+\!\setminus\!\mathcal{E}^+$:
\begin{equation}\label{eq:domain_of_HaY}
\mathcal{D}(-\Delta_{\alpha,Y})\;=\;\Big\{g=F_z + \sum_{j,k=1}^N (\Gamma_{\alpha,Y}(z)^{-1})_{jk} \, F_z(y_k)
{\mathcal{G}}_{z}^{y_j}  
\,,\,F_z \in H^2(\mathbb{R}^3) \Big\}\,.
\end{equation} 
Equivalently, for any $z\in\mathbb{C}^+\!\setminus\!\mathcal{E}^+$,
\begin{equation}\label{eq:domain_of_HaY_bc}
\mathcal{D}(-\Delta_{\alpha,Y})\;=\;\left\{g= F_z \\
+\sum_{j=1}^N c_j\,{\mathcal{G}}_{z}^{y_j}\left|\!\!
\begin{array}{c}
F_z \in H^2(\mathbb{R}^3) \\
(c_1,\dots,c_N)\in\mathbb{C}^N \\
\begin{pmatrix} F_z(y_1) \\ \vdots \\ F_z(y_N)\end{pmatrix}=\Gamma_{\alpha,Y}(z)\begin{pmatrix} c_1 \\ \vdots \\ c_N\end{pmatrix}
\end{array}\!\!\!\right.\right\}\,.
\end{equation}
At fixed $z$, the decompositions above are unique.
\item With respect to the decompositions \eqref{eq:domain_of_HaY}-\eqref{eq:domain_of_HaY_bc}, one has
\begin{equation}\label{eq:action_of_H}
 (-\Delta_{\alpha,Y}-z^2)\,g\;=\;(-\Delta -z^2)\,F_z\,.
\end{equation}
\item For $z\in\mathbb{C}^+\!\setminus\!\mathcal{E}^+$, we have the resolvent identity
\begin{equation}\label{eq:resolvent_identity}
(-\Delta_{\alpha,Y} -z^2)^{-1} -(-\Delta-z^2)^{-1} \;=\; 
\sum_{j,k=1}^N  (\Gamma_{\alpha,Y}(z)^{-1})_{jk} \,|\mathcal{G}_{z}^{y_j}\rangle\langle
\overline{\mathcal{G}_{z}^{y_k}}|\,.
\end{equation}
\item The spectrum $\sigma(-\Delta_{\alpha,Y})$ of 
$-\Delta_{\alpha,Y}$ consists of at most 
$N$ non-positive eigenvalues and the absolutely continuous 
part $\sigma_{\mathrm{ac}}(-\Delta_{\alpha,Y})=[0,\infty)$, the singular continuous 
spectrum is absent. 
\end{enumerate}
\end{proposition}

Parts (i) and (ii) of Proposition \ref{thm:general_properties} above originate from \cite{Grossmann-HK-Mebkhout-1980_CMPperiodic} and are discussed in \cite[Theorem II.1.1.3]{albeverio-solvable}, in particular \eqref{eq:domain_of_HaY_bc} is highlighted in
\cite{DFT-brief_review_2008}. Part (iii) was first proved in \cite{Grossmann-HK-Mebkhout-1980,Grossmann-HK-Mebkhout-1980_CMPperiodic} (see also \cite[equation (II.1.1.33)]{albeverio-solvable}). Part (iv) is discussed in \cite[Theorem II.1.1.4]{albeverio-solvable}, where it is stated that $\sigma_p(-\Delta_{\alpha,Y})\subset(-\infty,0)$. An errata at the end of the monograph (see also \cite{dellapana,gri_nov}) specifies that a zero eigenvalue embedded in the continuous spectrum can actually occur: in fact for every $N\geq 2$ one can find a configuration $Y$ of the $N$ centers, and coupling parameters $\alpha_1,\ldots \alpha_N$ such that $0\in\sigma_p(-\Delta_{\alpha,Y})$ -- see the discussion in Section \ref{sec:tre}.

Next, let us discuss in detail the spectral properties of $-\Delta_{\alpha,Y}$, whose resolvent is characterized by \eqref{eq:resolvent_identity} as an explicit rank-$N$ perturbation of the free resolvent. For negative eigenvalues, the situation is well-understood \cite[Theorem II.1.1.4]{albeverio-solvable}.

\begin{proposition}\label{thm:spectral_many_center}
There is a one to one correspondence between 
the poles $\ii\lambda\in\mathcal{E}^+$ of 
$\Gamma_{\alpha,Y}(z)^{-1}$ and the 
negative eigenvalues $-\lambda^2$ of $-\Delta_{\alpha,Y}$, counting the 
multiplicity. The eigenfunctions associated to 
the eigenvalue $-\lambda^2<0$ have the form
\[
 \psi\;=\;\sum_{j=1}^N c_j\,\mathcal{G}_{\ii\lambda}^{y_j},
\]
where $(c_1,\dots,c_N)\in\,Ker\,\Gamma_{\alpha,Y}(\ii\lambda)\setminus\{0\}$. 
\end{proposition}

Our main purpose is to analyze the spectral behavior of $-\Delta_{\alpha,Y}$ at $z=0$, and more generally when $z$ approaches the real line. The starting point is a classical version of the Limiting Absorption Principle for the free Laplacian. Given $\sigma>0$, we consider the Banach space
\begin{equation}
\mathbf{B}_{\sigma}:=\mathcal{B}(L^2(\R^3,\langle x\rangle^{2+\sigma}dx);L^2(\R^3,\langle x\rangle^{-2-\sigma}dx)),
\end{equation}
where $\langle x\rangle:=\sqrt{1+|x|^2}$, and $\mathcal{B}(X;Y)$ denotes the space of linear bounded operators from $X$ to $Y$. We have the following result \cite{Agmon, Kuroda,Jensen-Kato}.

\begin{proposition}[Limiting Absorption Principle for $-\Delta$]\label{LAP_free}
Let $\sigma>0$. For any $z\in\C^+$, we have $(-\Delta-z^2)^{-1}\in\mathbf{B}_{\sigma}$. Moreover, the map $\C^+\ni z\mapsto (-\Delta -z^2)^{-1}\in \mathbf{B}_{\sigma}$ can be continuously extended to the real line.
\end{proposition}

Owing to the resolvent formula \eqref{eq:resolvent_identity}, and observing that for any $z\in\C^+\cup\R$ the projectors $|\mathcal{G}_{z}^{y_j}\rangle\langle
\overline{\mathcal{G}_{z}^{y_k}}|$ belong to $\mathbf{B}_{\sigma}$, it is easy to deduce that also $-\Delta_{\alpha,Y}$ satisfies a Limiting Absorption Principle.

\begin{proposition}[Limiting Absorption Principle for $-\Delta_{\alpha,Y}$]\label{LAP}
Let $\sigma>0$. For every $z\in\C^+\setminus\mathcal{E}^+$, we have $(-\Delta_{\alpha,Y}-z^2)^{-1}\in\mathbf{B}_{\sigma}$. Moreover, the map
$$\C^+\setminus\mathcal{E}^+\ni z\mapsto (-\Delta_{\alpha,Y} -z^2)^{-1}\in \mathbf{B}_{\sigma}$$
can be continuously extended to $\R\setminus\mathcal{E}^0$.
\end{proposition}

Our main result is a resolvent expansion in a neighborhood of $z=0$, which in view of the previous discussion is the only possible singular point on the real line for the map $z\mapsto(-\Delta_{\alpha,Y}-z^2)^{-1}\in\mathbf{B}_{\sigma}$.

\begin{theorem}\label{th:main}
In a (real) neighborhood of $z=0$, we have the expansion
\begin{equation}\label{main_exp}
(-\Delta_{\alpha,Y}-z^2)^{-1}=\frac{R_{-2}}{z^2}+\frac{R_{-1}}{z}+R_0(z),
\end{equation}
where $R_{-2},\,R_{-1}\in \mathbf{B}_{\sigma}$ and $z\mapsto R_0(z)$ is a continuous $\mathbf{B}_{\sigma}$-valued map.
Moreover, $R_{-2}\neq 0$ if and only if zero is an eigenvalue for $-\Delta_{\alpha,Y}$.
\end{theorem}

\begin{remark}\label{tocco}
For Schr\"odinger operators of the form $-\Delta + V$, the Limiting Absorption Principle and the analogous of Theorem \ref{th:main} can be proved under suitable short-range assumptions on the scalar potential $V$ (see e.g. the classical papers \cite{Agmon,Jensen-Kato}). In this case, moreover, it is well-known that $R_{-1}\neq 0$ if and only if there exists a generalized eigenfunction at $z=0$ (a \emph{zero-energy resonance} for $-\Delta + V$), namely a function $\psi\in L^2(\R^3,\langle x\rangle^{-1-\sigma}dx)\setminus L^2(\R^3)$, for any $\sigma>0$, which satisfies $(-\Delta + V)\psi=0$ as a distributional identity on $\R^3$. As it will be clear from the proof of Theorem \ref{th:main}, a similar characterization holds true also for $-\Delta_{\alpha,Y}$ (see Remark \ref{rema_riso}).
\end{remark}

\section{Asymptotic for $\Gamma_{\alpha,Y}(z)^{-1}$ as $z\to 0$}\label{sedia}
We fix $N\geq 1$, $\alpha\in\R^N$ and  $Y\subseteq\R^3$, and we set $\Gamma(z):=\Gamma_{\alpha,Y}(z)$. 

We shall use the notation $O(z^k)$, $k\in\mathbb{Z}$, to denote a meromorphic $M^N(\mathbb{C})$-valued function whose Laurent expansion in a neighborhood of $z=0$ contains only terms of degree $\geq k$. In particular, $O(1)$ denotes an analytic map in a neighborhood of $z=0$. We also write $\Theta(z^k)$ to denote a function of the form $Az^{k}$, with $A\in\,M^{N}(\mathbb{C})\setminus\{0\}$. 

In a neighborhood of $z=0$, we can expand
$$\Gamma(z)=\Gamma_0+z\Gamma_1+z^2\Gamma_2+O(z^3).$$
Explicitly, we have
$$(\Gamma_0)_{jk}=\alpha_{j}\delta_{jk}-\mathcal{G}_0^{y_jy_k},\quad (\Gamma_1)_{jk}=(4\pi i)^{-1},\quad (\Gamma_2)_{jk}=(8\pi)^{-1}|y_j-y_k|.$$
In particular, $\Gamma_0$, $\Gamma_2$ are real symmetric matrices, while $\Gamma_1$ is skew-Hermitian, i.e. $\Gamma_1^*=-\Gamma_1$. Our aim is to characterize the small $z$ behavior of $\Gamma(z)^{-1}$. Preliminary, we recall the following useful result due to Jensen and Nenciu \cite{JN}.

\begin{lemma}[Jensen-Nenciu]\label{le:JN}
Let $A$ be a closed operator  in a Hilbert space $\mathcal{H}$ and $P$ a projection, such that $A+P$ has a bounded inverse. Then $A$ has a bounded inverse if and only if
$$B=P-P(A+P)^{-1}P$$
has a bounded inverse in $P\mathcal{H}$, and in this case
$$A^{-1}=(A+P)^{-1}+(A+P)^{-1}P(B \upharpoonright P\mathcal{H})^{-1}P(A+P)^{-1}.$$
\end{lemma}

We can state now the main result of this Section.

\begin{proposition}\label{pr:exp_inv_gamma}
In a neighborhood of $z=0$ we have the Laurent expansion
\begin{equation}\label{laure}
\Gamma(z)^{-1}=\frac{A_{-2}}{z^2}+\frac{A_{-1}}{z}+O(1),
\end{equation}
where $A_{-2},\,A_{-1}\in M^{N}(\mathbb{C})$. Moreover, 
\begin{enumerate}
\item[(i)] $A_{-2}\neq 0$ if and only if $\,Ker\,\Gamma_0\,\cap\,Ker\,\Gamma_1 \neq \{0\}$,
\item[(ii)] $A_{-1}\neq 0$ if and only if $\,Ker\,\Gamma_0 \not\subseteq\,Ker\,\Gamma_1$.
\end{enumerate}
\end{proposition}

\begin{proof} If $\Gamma_0=\Gamma(0)$ is non-singular, then $\Gamma(z)^{-1}$ is analytic in a sufficiently small neighborhood of $z=0$. Assume now that $\Gamma_0$ is singular. Let us distinguish two cases:\\

\textbf{Case 1:} $Ker\,\Gamma_0\,\cap\,Ker\,\Gamma_1 = \{0\}$. Let us set $\Gamma_{\leq 1}(z):=\Gamma_0+z\Gamma_1$, and observe that for $z$ small enough, $z\neq 0$, the matrix $\Gamma_{\leq 1}(z)$ is non-singular. Suppose indeed that $\Gamma_{\leq 1}(z)v=0$ for some $v\in\C^N$. If $\Gamma_0v\neq 0$, then for small $z$ we also have $\Gamma_{\leq 1}(z)v\neq 0$, a contradiction. Hence $\Gamma_0 v=0$, which for $z\neq 0$ implies $\Gamma_1 v=0$, and using the hypothesis $Ker\,\Gamma_0\,\cap\,Ker\,\Gamma_1 = \{0\}$ we deduce that $v=0$. Observe also that for $z$ small enough, $z\neq 0$, the matrix $\Gamma(z)$ in non-singular, with $\Gamma(z)^{-1}=\Gamma_{\leq 1}(z)^{-1}+ O(1)$. 

In order to invert $\Gamma_{\leq 1}(z)$, we use the Jensen-Nenciu Lemma. Let $P:\C^N\to\C^N$ be the orthogonal projection onto $Ker\,\Gamma_0$. Since $\Gamma_0^*=\Gamma_0$, we have that $\Gamma_0+P$ is non-singular, whence the same is $\Gamma_{\leq 1}(z)+P$ for small $z$, with $(\Gamma_{\leq 1}(z)+P)^{-1}=O(1)$. More precisely,
\begin{equation}\label{neon}
\begin{split}(\Gamma_{\leq 1}(z)+P)^{-1}&=[I+z(\Gamma_0+P)^{-1}\Gamma_1]^{-1}[\Gamma_0+P]^{-1}\\
&=[I-z(\Gamma_0+P)^{-1}\Gamma_1][\Gamma_0+P]^{-1}+O(z^2).
\end{split}
\end{equation}
By Lemma \ref{le:JN} we get
\begin{equation}\label{primaex}
\begin{split}
&\Gamma_{\leq 1}(z)^{-1}=(\Gamma_{\leq 1}(z)+P)^{-1}\\
&\quad+(\Gamma_{\leq 1}(z)+P)^{-1}P\Big(\big(P-P(\Gamma_{\leq 1}(z)+P)^{-1}P\big) \upharpoonright P\,\C^N\Big)^{-1}P(\Gamma_{\leq 1}(z)+P)^{-1}.
\end{split}
\end{equation}
Observe that $(\Gamma_0+P)^{-1}P=P$, and since $\Gamma_0^*=\Gamma_0$ we also have $P(\Gamma_0+P)^{-1}=P$. Using these relations and  \eqref{neon}, we compute
$$P-P(\Gamma_{\leq 1}(z)+P)^{-1}P=zP\Gamma_1P+O(z^2).$$
Substituting into \eqref{primaex} we obtain
\begin{equation}
\begin{split}
\Gamma_{\leq 1}^{-1}(z)&=(\Gamma_{\leq 1}(z)+P)^{-1}\\
&+(\Gamma_{\leq 1}(z)+P)^{-1}P\Big(\big(z\,P\Gamma_1P\upharpoonright P\,\C^N\big)^{-1}+O(1)\Big)P(\Gamma_{\leq 1}(z)+P)^{-1}\\
&=z^{-1}P(P\Gamma_1P\upharpoonright P\,\C^N)^{-1}P+O(1)=\Theta(z^{-1})+O(1).
\end{split}
\end{equation}

\textbf{Case 2:} $Ker\,\Gamma_0\,\cap\,Ker\,\Gamma_1 \neq \{0\}$.
We start by proving that $\,Ker\,\Gamma_1\,\cap\,Ker\,\Gamma_2\, = \{0\}$. Since $\Gamma_2$ is real symmetric, and $\Gamma_1$ is purely imaginary and skew-symmetric, it is sufficient to show that the quadratic form associated to $\Gamma_2$ is strictly negative on 
$$(\,Ker\,\Gamma_1\cap\R^N)\setminus\{0\}=\big\{v\in\R^N\setminus\{0\}\,|\,v_1+\ldots +v_N=0\big\}.$$ 
To this aim, we prove preliminary that for any $v\in\R^N$, with $v_1+\ldots +v_N=0$, 
\begin{equation}\label{bressa}
\langle\Gamma_2v,v\rangle:=(8\pi)^{-1}\sum_{1\leq j,k\leq N}|y_j-y_k|v_jv_k \leq 0.
\end{equation}

The key point is to use the so called \emph{averaging trick}. By rotational and scaling invariance, we can see that there exists a positive constant $c$ such that, for any $y\in\R^3$,
$$\int_{S^2}|\langle w,y\rangle|dw=c|y|.$$
It follows that
\begin{equation}\label{oronca}
(8\pi)^{-1}\sum_{1\leq j,k\leq N}|y_j-y_k|v_jv_k =(8\pi c)^{-1}\int_{S^2}\sum_{1\leq j,k\leq N}|\langle w,y_j-y_k\rangle|v_jv_kdw,
\end{equation}
and then it is sufficient to prove that, for a fixed $w\in S^2$,
$$\sum_{1\leq j,k\leq N}|\tilde{y}_j-\tilde{y}_k|v_jv_k\leq 0,$$
where we set $\tilde{y}_j:=\langle w,y_j\rangle$ for $j=1,\ldots, N$. We have
\begin{equation}\label{khy}
\begin{split}
\sum_{1\leq j,k\leq N}|\tilde{y}_j-\tilde{y}_k|v_jv_k&=2\sum_{1\leq j,k\leq N}\max{\{\tilde{y}_j-\tilde{y}_k,0\}}v_jv_k\\
&=2\int_{t\in\R}\sum_{1\leq j,k\leq N}[\tilde{y}_k<t<\tilde{y}_j]v_jv_k,
\end{split}
\end{equation}
where we use the Iverson bracket notation $[P]$, which equals $1$ if the statement $P$ is true and 0 if it is false. So it is enough to prove that, for almost every $t\in\R$, 
$$ \sum_{\tilde{y}_k<t<\tilde{y}_j}v_jv_k\leq 0.$$
For every $t\in\R\setminus\{\tilde{y}_1,\ldots \tilde{y}_N\}$, define $J_t:=\{j\,\big|\,\tilde{y}_j>t\}$, $K_t:=\{k\,\big|\,\tilde{y}_k<t\}$. We have
\begin{equation}\label{none}
\sum_{\tilde{y}_k<t<\tilde{y}_j}v_jv_k=\sum_{j\in J_t,k\in K_t}v_jv_k=\Big(\sum_{j\in J_t}v_j\Big)\Big(\sum_{k\in K_t}v_k\Big)=-\Big(\sum_{j\in J_t}v_j\Big)^2\leq 0,
\end{equation}
where we use, in the last equality, the hypothesis $v_1+\ldots +v_N=0$.

Assume now that we have the equality in \eqref{bressa}, for a suitable vector $v\in\R^N$ with $v_1+\ldots v_N=0$. It follows from the identity \eqref{oronca} that for almost every $w\in S^2$
\begin{equation}\label{oraora}
\sum_{1\leq j,k\leq N}|\langle w,y_j-y_k\rangle|v_jv_k=0.
\end{equation}
In particular, we can choose $w\in S^2$ satisfying \eqref{oraora}, and such that the quantities $\tilde{y}_j=\langle w,y_j\rangle$ are pairwise distinct, say $\tilde{y}_1>\tilde{y}_2>\ldots >\tilde{y}_N$.
Owing to \eqref{khy}-\eqref{none}, we deduce that
\begin{equation}\label{grenoble}
\sum_{\tilde{y}_k<t<\tilde{y}_j}v_jv_k=0,
\end{equation}
for every $t$ in a full-measure set $\mathcal{T}\subset\R$. In particular, choosing $t_1,\ldots t_{N-1}\in\mathcal{T}$, with $t_n\in(\tilde{y}_{n+1},\tilde{y}_n)$ for $n=1,\ldots ,N-1$, we obtain from \eqref{grenoble} and \eqref{none} that
$$\sum_{j=1}^nv_j=0\quad\forall\,n\in\{1,\ldots ,N\}.$$
This implies $v=0$, concluding the proof that $\,Ker\,\Gamma_1\,\cap\,Ker\,\Gamma_2\, = \{0\}$.

Now, let us set $\Gamma_{\leq 2}(z):=\Gamma_{\leq 1}(z)+z^2\Gamma_2$. Arguing as in Case 1, and using the property $\,Ker\,\Gamma_1\,\cap\,Ker\,\Gamma_2\, = \{0\}$, we deduce that for $z$ small enough, $z\neq 0$, the matrix $\Gamma_{\leq 2}(z)$ is non-singular. In particular, for $z\neq 0$ small enough, also $\Gamma(z)$ is non-singular, with $\Gamma(z)^{-1}=\Gamma_{\leq 2}(z)^{-1}+ O(1)$. 

In order to invert $\Gamma_{\leq 2}(z)$, we use the Jensen-Nenciu Lemma. Let $P:\C^N\to\C^N$ be the orthogonal projection onto $Ker\,\Gamma_0\,\cap\,Ker\,\Gamma_1$. Owing to the relations $\Gamma_0^*=\Gamma_0$, $\Gamma_1^*=-\Gamma_1$, we deduce that for $z$ small enough $\Gamma_{\leq 1}(z)+P$ is non-singular, with
\begin{equation}\label{megliocopiare}
(\Gamma_{\leq 1}(z)+P)^{-1}=\begin{cases}
\Theta(z^{-1}) + O(1)&Ker\,\Gamma_0 \not\subseteq\,Ker\,\Gamma_1\\
O(1)&Ker\,\Gamma_0\subseteq\,Ker\,\Gamma_1
\end{cases}.
\end{equation}
For small $z$, also $\Gamma_{\leq 2}(z)+P$ is non-singular, with 
$$(\Gamma_{\leq 2}(z)+P)^{-1}=(\Gamma_{\leq 1}(z)+P)^{-1}+O(1).$$ 
With similar computations as in Case 1, we get
\begin{equation}
\begin{split}
\Gamma_{\leq 2}(z)^{-1}&=(\Gamma_{\leq 2}(z)+P)^{-1}+z^{-2}P(P\Gamma_2P\upharpoonright P\,\C^N)^{-1}P+O(1)\\
&=\begin{cases}
\Theta(z^{-2})+\Theta(z^{-1}) + O(1)&Ker\,\Gamma_0 \not\subseteq\,Ker\,\Gamma_1\\
\Theta(z^{-2})+O(1)&Ker\,\Gamma_0\subseteq\,Ker\,\Gamma_1
\end{cases}.
\end{split}
\end{equation}
Expansion \eqref{laure} is thus proved in any case. Moreover, statements (i) and (ii) easily follow from the discussion above.
\end{proof}

\section{Proof of the main Theorem}\label{sepm}
This Section is devoted to the proof of Theorem \ref{th:main}. Let us fix $N\geq 1$, $\alpha\in\R^N$ and $Y\subseteq\R^3$, and set $\Gamma(z):=\Gamma_{\alpha,Y}(z)$. Preliminary, observe that the low-energy expansion \eqref{main_exp} follows by combining the resolvent formula \eqref{eq:resolvent_identity} with the small $z$ expansion \eqref{laure} for $\Gamma(z)^{-1}$. We prove now that $R_{-2}\neq 0$ if and only if $0\in\sigma(-\Delta_{\alpha,Y})$, which in view of Proposition \ref{pr:exp_inv_gamma}, part (i), is equivalent to prove that $Ker\,\Gamma_0\,\cap\,Ker\,\Gamma_1 \neq \{0\}$ if and only if $0\in\sigma(-\Delta_{\alpha,Y})$.

Suppose first that there exists $c=(c_1,\ldots ,c_N)\neq 0\in\,Ker\,\Gamma_0\,\cap\,Ker\,\Gamma_1$. We are going to show that the non-zero function
\begin{equation}\label{def_psi}
\psi:=\sum_{j=1}^Nc_j\mathcal{G}_0^{y_j}
\end{equation}
belongs to $Ker\,(-\Delta_{\alpha,Y})$. First of all, observe that the condition $\Gamma_1\,c=0$ is equivalent to $c_1+\ldots +c_N=0$, which implies $\psi\in L^2(\R^3)$.

Let us fix $z\in\C^+\setminus\mathcal{E}^+$, and write
$$\psi=F_z+\sum_{j=1}^Nc_j\mathcal{G}_z^{y_j},$$
where
$$F_z:=\sum_{j=1}^Nc_j(\mathcal{G}_0^{y_j}-\mathcal{G}_z^{y_j}).$$
Observe that $F_z\in H^2(\R^3)$. Moreover, for every $k\in\{1,\ldots ,N\}$,
$$F_z(y_k)=\sum_{j=1}^Nc_j(\mathcal{G}_0^{y_j\,y_k}-\mathcal{G}_z^{y_j,y_k})=\sum_{k=1}^N\Gamma(z)_{kj}c_j,$$
where in the second equality we use that $\Gamma_0\,c=\Gamma_1\,c=0$. By virtue of representation \eqref{eq:domain_of_HaY_bc}, we conclude that $\psi\in\mathcal{D}(-\Delta_{\alpha,Y})$. Moreover, formula \eqref{eq:action_of_H} yields
$$-\Delta_{\alpha,Y}\psi=(-\Delta-z^2)F_z+z^2\sum_{j=1}^Nc_j\mathcal{G}_z^{y_j}=\sum_{j=1}^Nc_j\big[(-\Delta-z^2)\mathcal{G}_z^{y_j}-\Delta \mathcal{G}_0^{y_j}\big]=0,$$
which shows that $\psi\in\,Ker(-\Delta_{\alpha,Y})$.

Let us discuss now the opposite implication. To this aim, consider a function $\psi\in\,Ker\,(-\Delta_{\alpha,Y})\setminus\{0\}$. For a fixed $z=i\lambda\in\C^+\setminus\mathcal{E}^+$, we can write
\begin{equation}\label{resilie}
\psi=F_{i\lambda}+\sum_{j=1}^Nc_j\mathcal{G}_{i\lambda}^{y_j},
\end{equation}
for some non-zero $F_z\in H^2(\R^3)$, where
$$c_j=\sum_{k=1}^N\Gamma(z)^{-1}_{jk}F_z(y_k).$$
Observe that the $c_j$'s are necessarily independent of $z$, since $\mathcal{G}_{i\lambda}^{y_j}\not\in H^2(\R^3)$ for any $j$. Moreover, the condition $\psi\in L^2(\R^3)$  implies $c_1+\ldots +c_n=0$, namely $\Gamma_1c=0$. Owing to \eqref{eq:action_of_H} and the representation \eqref{resilie}, the relation $-\Delta_{\alpha,Y}\psi=0$ is equivalent to
\begin{equation}\label{marilori}
-\Delta\,F_{i\lambda}=\lambda^2\sum_{j=1}^Nc_j\mathcal{G}_{i\lambda}^{y_j}.
\end{equation}
We show now that $\|F_{i\lambda}\|_{H^2}\to 0$ as $\lambda\downarrow 0$, whence also $F_{i\lambda}\to 0$ as $\lambda\downarrow 0$, uniformly on compact subsets of $\R^3$. This implies 
$$\Gamma_0\,c=\lim_{\lambda\downarrow 0}\,\Gamma(i\lambda)c=0,$$
and the identity
$$\psi=\sum_{j=1}^Nc_j\mathcal{G}_{0}^{y_j},$$
which conclude the proof.

In order to show that $\|F_{i\lambda}\|_{H^2}\to 0$ as $\lambda\downarrow 0$, we start with the estimate
\begin{equation}\label{trivialK}
\|\Delta F_{i\lambda}\|_{L^2}=\|\lambda^2\Delta(-\Delta+\lambda^2)^{-1}\psi\|_{L^2}\leq \lambda^2\|\psi\|_{L^2}.
\end{equation}
Observe moreover that $\widehat{F_{i\lambda}}(p)=\lambda^2(p^2+\lambda^2)^{-1}\widehat{\psi}(p)$. By dominate convergence we get $\|F_{i\lambda}\|_{L^2}=o(1)$, which combined with \eqref{trivialK} yields $\|F_{i\lambda}\|_{H^2}=o(1)$, as desired.

\begin{remark}\label{rema_riso}
By Proposition \ref{pr:exp_inv_gamma}(ii), there is a $\Theta(z^{-1})$ term in the expansion of $\Gamma(z)^{-1}$ at $z=0$ if and only if there exists $c\in\R^n$ such that $\Gamma_0c=0$, $\Gamma_1c\neq 0$. In this case, the function defined by \eqref{def_psi} belongs to $L^2(\R^3,\langle x\rangle^{-1-\sigma}dx)\setminus L^2(\R^3)$, for any $\sigma>0$, and formally satisfies $-\Delta_{\alpha,Y}\psi=0$, whence $\psi$ can be interpreted as a zero energy resonance for $-\Delta_{\alpha,Y}$. Hence, as anticipated in Remark \ref{tocco}, we have that $R_{-1}\neq 0$ in expansion \eqref{main_exp} if and only if there exists a zero energy resonance, analogously to the case of classical Schr\"odinger operators.
\end{remark}

\section{Occurrence and multiplicity of zero energy obstructions}\label{sec:tre}

In this Section we discuss the occurrence and the multiplicity of obstructions at zero energy for the resolvent of $-\Delta_{\alpha,Y}$, depending on the choice of the set $Y$ of centers of interactions and the coupling parameters $\alpha_1, \ldots \alpha_N$.

In the single center case, it is easy to check that the only possible obstruction at $z=0$ is a resonance, attained if and only if $\alpha=0$.
In general, for any $N$ and for any given configuration of the centers, there exists a measure zero set of choices of the parameters $\alpha_1,\ldots ,\alpha_N$ which leads to a zero-energy resonance. By means of the discussion in Section \ref{sedia} and Section \ref{sepm}, we can define the multiplicity of the zero-energy resonance as 
$$r_{\alpha,Y}:=dim\,(\,Ker\,\Gamma_0)-dim\,(\,Ker\,\Gamma_0\,\cap\,Ker\,\Gamma_1).$$
We conjecture that, as $N$ increases, one can find $Y$ and $\alpha$ such that $r_{\alpha,Y}$ becomes arbitrarily large.

As anticipated in Section \ref{secuno}, when $N=2$ we can find a simple zero eigenvalue by choosing $\alpha_1=\alpha_2=-(4\pi d)^{-1}$, where $d$ is the distance between the two centers. For $N\geq 3$, a zero eigenvalue occurs for specific geometric configurations of the centers of interactions and for a measure zero set of choices of $\alpha_1,\ldots ,\alpha_N$. Owing to the discussion in Section \ref{sedia} and Section \ref{sepm}, the multiplicity of the zero eigenvalue is given by
$$e_{\alpha,Y}:=dim\,Ker\,(-\Delta_{\alpha,Y})=dim\,(\,Ker\Gamma_0\cap\,Ker\,\Gamma_1).$$
Let us discuss now the maximal possible value for $e_{\alpha,Y}$ as the number of centers of interactions increases.

\begin{itemize}
\item $N=3$. We can take $Y$ as the vertices of an equilateral triangle of side-length one, and $\alpha_1=\alpha_2=\alpha_3=-(4\pi)^{-1}$. With this choice we get $e_{\alpha,Y}=2$.

\item $N=4$. We can take $Y$ as the vertices of a regular tetrahedron of side-length one, and $\alpha_1=\alpha_2=\alpha_3=\alpha_4=-(4\pi)^{-1}$. With this choice we get $e_{\alpha,Y}=3$.

\item $N=5$. Observe that we can not find five points in $\R^3$ with constant pairwise distances. It easily follows that the maximal value for $e_{\alpha,Y}$ is still three.
\end{itemize}

One could conjecture that for $N\geq 4$ the maximal value of $e_{\alpha,Y}$ is three. Nevertheless, it is also conceivable that for large $N$ there exist complicated geometrical configurations which lead to a higher multiplicity. Such kind of mechanism is well-known in similar contexts.
Consider, for example, the problem in combinatorics to determine the chromatic number of the \emph{unit distance graph} on $\R^3$, that is the graph with vertices set $V=\R^3$ and edges set $E=\{(x,y)\in\R^3\times\R^3\,|\,|x-y|=1\}$. Owing to the compactness principle by De Bruijn and Erd\H{o}s \cite{Erdos} this is equivalent, under the axiom of choice, to determine the highest chromatic number of a finite graph embedded in $\R^3$ in such a way all its edges have length one. For a graph with $N$ vertices, we have the following situation:

\begin{itemize}
\item $N=3$. We can consider an equilateral triangle of side-length one, which has chromatic number three.
\item $N=4$. We can consider a regular tetrahedron of side-length one, which has chromatic number four.
\item $N=5$. The highest possible chromatic number is still four.
\item $N=14$. There is a configuration of $14$ points in $\R^3$, the Moser-Raiskii spindle, with chromatic number five \cite{raiskii,szewor}.
\item For large $N$, the highest possible chromatic number is known to be between $6$ and $12$ \cite{Nechushtan, radototh, coulson}.
\end{itemize}

It is evident that there are similarities between the two problems, and it would be interesting to understand if they are actually related. In particular, one may take $Y$ as the vertices of the Moser-Raiskii spindle and wondering whether there exists $\alpha=(\alpha_1,\ldots ,\alpha_{14})$ such that $e_{\alpha,Y}=4$.

\section*{Acknowledgments}
The author acknowledges an anonymous referee for the useful suggestions, and for pointing out relevant references.


\begin{thebibliography}{99.}%
%
%
%
%







\bibitem{Agmon} {\sc S.~Agmon}, {\em{Spectral properties of Schrodinger operators and scattering theory}}, Ann. Scuola Norm. Sup. Pisa Cl. Sci., 2 (1975), pp.~151--218.

\bibitem{Albeverio-Fenstad-HoeghKrohn-1979_singPert_NonstAnal}
{\sc S.~Albeverio, J.~E. Fenstad, and R.~H{\o}egh-Krohn}, {\em {Singular
  perturbations and nonstandard analysis}}, Trans. Amer. Math. Soc., 252
  (1979), pp.~275--295.

\bibitem{Albeverio-Figari}
{\sc S.~Albeverio and R.~Figari}, {\em{Quantum fields and point interactions}}, Rend. Mat. Appl. \textbf{39} (2018), 161--180.

\bibitem{albeverio-solvable}
{\sc S.~Albeverio, F.~Gesztesy, R.~H{\o}egh-Krohn, and H.~Holden}, {\em
  {Solvable {M}odels in {Q}uantum {M}echanics}}, {Texts and Monographs in
  Physics}, Springer-Verlag, 2012.

\bibitem{Berezin-Faddeev-1961}
{\sc F.~A.~Berezin and L.~D.~Faddeev}, {\em {A Remark on Schr\"odinger's equation with
  a singular potential}}, Doklady Akademii Nauk Ser. Fiz., 137 (1961), pp.~1011--1014 (In Russian); English translation: Sov. Math. Dokl., 2 (1961), pp.~372--375.

\bibitem{Cornean-MY}
{\sc H.~D.~Cornean, A.~Michelangeli, and K.~Yajima}, {\em {Two dimensional Schr\"odinger operators with point interactions: threshold expansions, zero modes and $L^p$-boundedness of wave operators}}, arXiv:1804.01297 (2018)

\bibitem{coulson}
{\sc D.~Coulson}, {\em{A $15$-colouring of $3$-space omitting distance one}}, Disc. Math., 256 (2002), pp. 83--90.



\bibitem{Dabrowski-Grosse-1985}
{\sc L.~Dabrowski and H.~Grosse}, {\em {On nonlocal point interactions in one,
  two, and three dimensions}}, J. Math. Phys., 26 (1985), pp.~2777--2780.

\bibitem{DAncona-Pierfelice-Teta-2006}
{\sc P.~D'Ancona, V.~Pierfelice, and A.~Teta}, {\em {Dispersive estimate for
  the {S}chr{\"o}dinger equation with point interactions}}, Math. Methods Appl.
  Sci., 29 (2006), pp.~309--323.

\bibitem{Erdos}
{\sc N.~G.~De Bruijn and P.~Erd\H{o}s}, {\em{A colour problem for infinite graphs and a problem in the theory of relations}}, Nederl. Akad. Wetensch. Proc. Ser. A, 54 (1951), pp. 371–373.

\bibitem{DFT-brief_review_2008}
{\sc G.~Dell'Antonio, R.~Figari, and A.~Teta}, {\em {A brief review on point
  interactions}}, in {Inverse problems and imaging}, vol.~1943 of {Lecture
  Notes in Math.}, Springer, Berlin, 2008, pp.~171--189.

\bibitem{DMSY-2017}
{\sc G.~Dell'Antonio, A.~Michelangeli, R.~Scandone, and K.~Yajima}, {\em
  {{$L^p$}-{B}ounded\-ness of {W}ave {O}perators for the {T}hree-{D}imensional
  {M}ulti-{C}entre {P}oint {I}nteraction}}, Ann. Henri Poincar{\'e}, 19 (2018),
  pp.~283--322.

\bibitem{dellapana}
{\sc G.~Dell'Antonio and G.~Panati}, {\em{A remark on the existence of zero-energy bound states for point interaction Hamiltonians}}, unpublished notes.

\bibitem{Galtbayar-Yajima-2019}
{\sc A.~Galtbayar and K.~Yakima}, {\em {On the approximation by regular potentials of Schr{\"o}\-dinger operators with point interactions}}, arXiv:1908.02936 (2019).

\bibitem{gri_nov}
{\sc P.~Grinevich and R.~G.~Novikov}, {\em {Multipoint scatterers with zero-energy bound states}}, Theor. Math. Phys., 193 (2017), pp.~1675--1679.


\bibitem{Grossmann-HK-Mebkhout-1980}
{\sc A.~Grossmann, R.~H{\o}egh-Krohn, and M.~Mebkhout}, {\em {A class of
  explicitly soluble, local, many-center {H}amiltonians for one-particle
  quantum mechanics in two and three dimensions. {I}}}, J. Math. Phys., 21
  (1980), pp.~2376--2385.

\bibitem{Grossmann-HK-Mebkhout-1980_CMPperiodic}
\leavevmode\vrule height 2pt depth -1.6pt width 23pt, {\em {The one particle
  theory of periodic point interactions. {P}olymers, monomolecular layers, and
  crystals}}, Comm. Math. Phys., 77 (1980), pp.~87--110.

\bibitem{Iandoli-Scandone-2017}
{\sc F.~Iandoli and R.~Scandone}, {\em {Dispersive estimates for
  Schr{\"o}dinger operators with point interactions in $\mathbb{R}^3$}}, in
  {Advances in Quantum Mechanics: Contemporary Trends and Open Problems},
  A.~Michelangeli and G.~Dell'Antonio, eds., {Springer INdAM Series, vol.~18},
  Springer International Publishing, pp.~187--199.

\bibitem{Jensen-Kato}
{\sc A.~Jensen and T.~Kato}, {\em {Spectral properties of {S}chr{\"o}dinger operators and time-decay of the wave functions}}, Duke. Math. J., 46
  (1979), pp.~583--611.

\bibitem{JN}
{\sc A.~Jensen and G.~Nenciu}, {\em {A unified approach to resolvent expansions at thresholds}}, Rev. in Mathe. Phys., 13 (2001), pp.~717--754

\bibitem{Kuroda}
{\sc S.~T.~Kuroda}, {\em{Introduction to Scattering Theory}}, Lecture Notes, Matematisk Institute, Aarhus University (1978).

\bibitem{Lax-Phi}
{\sc P.~D.~Lax and R.~S.~Phillips}, {\em {Scattering theory}}, vol. 26, Academic press (1990).



\bibitem{MO-2016}
{\sc A.~Michelangeli and A.~Ottolini}, {\em {On point interactions realised as
  {T}er-{M}artirosyan-{S}kornyakov {H}amiltonians}}, Rep. Math. Phys., 79
  (2017), pp.~215--260

\bibitem{MS-2020}
{\sc A.~Michelangeli and R.~Scandone}, {\em {On real resonances for three-dimensional Schr\"o\-dinger operators with point interactions}}, arXiv:2002.07787 (2020)

\bibitem{Nechushtan}
{\sc O.~Nechushtan}, {\em{On the space chromatich number}}, Disc. Math. 256 (2002), pp. 499--507.

\bibitem{Posilicano2000_Krein-like_formula}
{\sc A.~Posilicano}, {\em {A {K}re\u\i n-like formula for singular
  perturbations of self-adjoint operators and applications}}, J. Funct. Anal.,
  183 (2001), pp.~109--147.

\bibitem{radototh}
{\sc R.~Radoi\v{c}i\'c and G.~T\'oth}, {\em{Note on the Chromatic number of the Space}}, Disc. Comput. Geometry. Algorithms and Combinatorics, 25 (2003), pp. 695--698.

\bibitem{raiskii}
{\sc D.~E.~Ra\v{i}ski\v{i}},  {\em{The realization of all distances in a decomposition of the space $\R^n$ into $n+1$ parts}}, Mat. Zametki 7 (1970), pp. 319--323 (in Russian); English translation: Math. Notes 7 (1970), pp. 194--196.

\bibitem{sjo_lecture}
{\sc J.~Sj\"ostrand}, {\em{Lectures on resonances}} (2002), sjostrand.perso.math.cnrs.fr/Coursgbg.pdf.


\bibitem{Sjo-Zwo}
{\sc J.~Sj\"ostrand and M.~Zworski}, {\em {Complex Scaling and the Distribution of Scattering Poles}}, Journ. Amer. Math. Soc., 4
  (1991), pp.~729--769.

\bibitem{szewor}
{\sc L.~A.~Sz\'{e}keley and N.~C.~Wormald}, {\em{Bounds on the measurable chromatic number on $\R^n$}}, Disc. Math., 75 (1-3) (1989), pp. 343--372.



\bibitem{Zorbas-1980}
{\sc J.~Zorbas}, {\em {Perturbation of self-adjoint operators by {D}irac
  distributions}}, J. Math. Phys., 21 (1980), pp.~840--847.



\end{thebibliography}
\end{document}